\documentclass[a4paper,12pt]{amsart}
\usepackage{amsmath,amsfonts,amssymb}
\usepackage{tikz}
\usepackage{multirow}
\usepackage{graphicx}
\usepackage{verbatim}
\usepackage{url}
\usepackage{hyperref}
\usepackage{float}
\floatstyle{plaintop}
\restylefloat{table}
\calclayout
\theoremstyle{plain}
\newtheorem{theorem}{Theorem}[section]
\newtheorem{lemma}[theorem]{Lemma}
\newtheorem{corollary}[theorem]{Corollary}
\newtheorem{proposition}[theorem]{Proposition}

\theoremstyle{definition}
\newtheorem*{definition}{Definition}

\theoremstyle{remark}

\newcommand{\C}{\mathbb{C}}

\newcommand{\Z}{\mathbb{Z}}

\newcommand{\Hy}{\mathbb{H}}
\newcommand{\Oc}{\mathcal{O}}
\newcommand{\Cb}{\mathbb{C}}

\begin{document}

\newpage
\title{On the volume  and the Chern-Simons invariant for the $2$-bridge knot orbifolds}
\author{Ji-Young Ham, Joongul Lee, Alexander Mednykh*, and Aleksey Rasskazov}
\address{Department of Science, Hongik University, 
94 Wausan-ro, Mapo-gu, Seoul,
 04066\\
   Korea} 
\email{jiyoungham1@gmail.com.}

\address{Department of Mathematics Education, Hongik University, 
94 Wausan-ro, Mapo-gu, Seoul,
04066\\
   Korea} 
\email{jglee@hongik.ac.kr}

\address{Sobolev Institute  of Mathematics,  pr. Kotyuga 4, Novosibirsk  630090  \\
Novosibirsk State University,  Pirogova 2, Novosibirsk  630090\\
 Russia.}
\email{mednykh@math.nsc.ru}
\thanks{*The author  was funded by the Russian Science Foundation (grant 16-41-02006).}

\address{Webster International University \\
146 Moo 5, Tambon Sam Phraya, Cha-am, Phetchaburi 76120\\
Tailand}
\email{arasskazov69@webster.edu}



\maketitle 

\markboth{ Ji-young Ham, Joongul Lee, Alexander Mednykh, and Aleksey Rasskazov } 
{On the volume  and the Chern-Simons invariant for the $2$-bridge knot orbifolds}
 
\begin{abstract}
We extend some part of the unpublished paper~\cite{MR2} written by Mednykh and Rasskazov. Using the approach indicated in this paper we derive the Riley-Mednykh polynomial for some family of the $2$-bridge knot orbifolds. As a result we obtain explicit formulae for the volume of cone-manifolds and the Chern-Simons invariant of orbifolds of the knot with Conway's notation $C(2n,4)$.
\end{abstract}
\maketitle

%%%%%%%%%%%%%%%%%%%%%%%%%%%%%%%%%%%%%%%%%%%%%%%%%%%%%%%%%%%%%%%%%%%
\section{Introduction}
By Mostow-Prasad rigidity, the two bridge knot complement has a unique hyperbolic structure if it has. Since the volume is a fundamental invariant, it has been studied quite a lot. 
But explicit volume formulae for hyperbolic cone-manifolds of knots and links are known a little. The volume formulae for hyperbolic cone-manifolds of the knot 
$4_1$~\cite{HLM1,K1,K2,MR1}, the knot $5_2$~\cite{M2}, the link $5_1^2$~\cite{MV1}, 
the link $6_2^2$~\cite{M1}, and the link $6_3^2$~\cite{DMM1} have been computed. In~\cite{HLM2}, a method of calculating the volumes of two-bridge knot cone-manifolds were introduced but without explicit formulae. In~\cite{HMP,HL1}, explicit volume formulae for hyperbolic cone-manifolds of twist knots and knots with Conway's notation $C(2n, 3)$ are computed. In~\cite{HLMR1}, explicit volume formulae for the link $7_3^2 (\alpha, \alpha)$ cone-manifolds are computed. In~\cite{Tran1, Tran2}, explicit volume formulae for double twist knot cone-manifolds and double twist link cone-manifolds are introduced without explicit computations.

Chern-Simons invariant~\cite{CS, Mey1} was defined to be a geometric invariant and became a topological invariant for hyperbolic two-bridge knots after the Mostow Rigidity Theorem~\cite{Mo1}.
Various methods of finding Chern-Simons invariant using ideal triangulations have been 
introduced~\cite{N1,N2,Z1,CMY1,CM1,CKK1} and
 implemented~\cite{SnapPy, Snap}. But, for orbifolds, to our knowledge, there does not exist a single convenient program which computes Chern-Simons invariant.
In~\cite{HLM2} a method of calculating the Chern-Simons invariants of two-bridge knot orbifolds were introduced but without explicit formulae. In~\cite{HL,HL2}, the Chern-Simons invariants of orbifolds of the twist knots and the knots with Conway's notation $C(2n, 3)$ are computed.
Similar approaches for $SU(2)$-connections can be found in~\cite{KK1} and for $\text{\textnormal{SL}}(2,C)$-connections in~\cite{KK2}. Explicit integral formulae for Chern-Simons invariants of the Whitehead link (the two component twist link) orbifolds and their cyclic coverings are presented in~\cite{A,A1}.
For explanations of cone-manifolds, you can refer to~\cite{CHK,T1,K1,P2,HLM1,PW,HMP}.

Let $\Oc(p/q,n)$ denote the orbifold with the singular set the hyperbolic
$2$-bridge knot with the cone-angle $2 \pi/n$. 
The geometry of the orbifolds $\Oc(p/q,n)$ has already been intensively investigated by many
authors. The combinatorial construction of fundamental polyhedra for orbifolds $\Oc(p/q,n)$ was suggested by Minkus \cite{min}.  Later it was discovered by Mednykh and Rasskazov \cite{MR1, MR2} that this topological construction can be successfully realized in hyperbolic, spherical and Euclidean spaces. Further development of this fruitful idea in the spaces of constant curvature was done in \cite{DMM1, KM, VR, Shm}. From \cite{HMTT} the above mentioned construction of the fundamental polyhedra became known as {\it butterfly polyhedra}. In  \cite{MSV} the geometrical ideas from \cite{MR1, MR2} were realized in  the five exotic Thurston geometries $\mathbb{S} \times\mathbb{R}, \mathbb{H}^2\times\mathbb{R},  \mathbb{SL}_2(\mathbb{R}), \mathbb{N}il, \text{  and  } \mathbb{S}ol.$ We use hyperbolic polyhedra in this paper. But, we also introduce Spherical and Euclidean polyhedra here since they come from the same topological construction.

 We derive the Riley-Mednykh polynomial responsible for the geometry of the fundamental polyhedron. Let $X_{q/p}$ be 
 the complement of a two bridge knot in $\mathbb{S}^3$.
 Let $R=\text{Hom}(\pi_1 (X_{q/p}), \text{SL}(2, \C))$. 
Given  a set of generators, $s,t$, of the fundamental group for 
$\pi_1 (X_{q/p})$, we define
 a set $R\left(\pi_1 (X_{q/p})\right) \subset \text{SL}(2, \C)^2 \subset \C^{8}$ to be the set of
 all points $(\rho(s),\rho(t))$, where $\rho$ is a
 representaion of $\pi_1 (X_{q/p})$ into $\text{SL}(2, \C)$. Let $S=\rho(s)$, $T=\rho(t)$ and  and $c \in \text{SL}(2, \C)$ which satisfies $cS=T^{-1}c$ and $c^2=-I$. Since the defining relation of 
 $\pi_1 (X_{q/p})$ gives the defining equation of $R\left(\pi_1 (X_{q/p})\right)$~\cite{R3}, $R\left(\pi_1 (X_{q/p})\right)$ is an affine algebraic set in $\C^{8}$ and can be identified with an algebraic set in $\C^{2}$ because the entries of $S$ and $T$ can be expressed in terms of two variables. The above fundamental polyhedron gives the natural choices for them, $A=\cot{\frac{\alpha}{2}}$ and $V=\cosh{d}$ where $\alpha$ is the cone angle along the knot and $d$ is the distance between two fixed axes of $S$ and $T$. The defining equation of the algebraic set in $\C^{2}$ is $\text{tr}(SWc)$ where $W=\rho(w)$ for $w$ in Proposition~\ref{prop:fundamentalGroup} (see Subsection~\ref{subsec:RM}). Since for any $W$, $\text{tr}(SWc)$ has $\text{tr}(Sc)$ as a factor, we define \emph{Riley-Mednykh polynomial} as $\text{tr}(SWc)/\text{tr}(Sc)$ up to powers of $\sin{\frac{\alpha}{2}}$. $\text{tr}(Sc)$ and $\sin{\frac{\alpha}{2}}$ give reducible representations. We can vary two choices. Then  the Riley-Mednykh polynomial comes from $\text{tr}(SWc)/\text{tr}(Sc)$ but sometimes we need to factor something else out to make it into a polynomial. In~\cite{HMP,HL}, instead of $A$ and $V$, we used $B=\cos{\frac{\alpha}{2}}$ and $V$. And 
 $\text{tr}(SWc)/\text{tr}(Sc)$ turns out to be a polynomial. Hence the Riley-Mednykh polynomial in this case is $\text{tr}(SWc)/\text{tr}(Sc)$. 
 In~\cite{HL1,HL2,HL3}, we used $M=e^{\frac{i \alpha}{2}}$ and $x=2-\text{tr}(ST)$. In this case, the Riley-Mednykh polynomial is $\text{tr}(SWc)/\text{tr}(Sc)$ up to powers of $M$. In~\cite{HLMR1}, we factored out some more and named it to be the Riley-Mednykh polynomial since we were interested in the geometric one.
  
 As an application, we present explicit formulae for the volume of cone-manifolds and the Chern-Simons invariant of orbifolds of the knot with Conway's notation $C(2n,4)$. We used $M$ and $x$ in this case. Among the equivalent knots, we consider $C(2n,4)$ as $K_{\frac{6n+1}{8n+1}}$. Hence the slope of $C(2n,4)$ is 
 $\frac{6n+1}{8n+1}$ (see Section~\ref{sec:knot} for definitions). Instead of working on complicated combinatorics of 3-dimensional ideal tetrahedra to find the volume and the Chern-Simons invariant of the hyperbolic orbifolds of the knot with Conway's notation $C(2n, 4)$, we deal with simple one dimensional singular loci. 
We use the Schl\"{a}fli formula and the Schl\"{a}fli formula for the generalized Chern-Simons function on the family of $C(2n,4)$ cone-manifold structures~\cite{HLM3}. With the normal precision of Mathematica, we could compute the volume and the Chern-Simons invariant of the hyperbolic orbifolds of the knot with Conway's notation $C(2n, 3)$ in~\cite{HL2}, but we couldn't for $C(2n,4)$ as $n$ gets large. In this paper, by elevating the precision to higher degree than the normal in Mathematica, we could  finally compute for $C(2n,4)$ with higher precision. With our Riley-Mednykh polynomial, locating the root corresponding to the geometric structures becomes easy.
%%%%%%%%%%%%%%%%%%%%%%%%%%%%%%%%%%%%%%%%%%%%%%%%%%%%%%%%%
%%%%%%%%%%%%%%%%%%%%%%%%%%%%%%%%%%%%%%%%%%%%%%%%%%%%%%%%%%
\section{$2$-bridge knots} 
\label{sec:knot}

The following theorem gives the classification of the $2$-bridge knots in normal forms. 

\begin{theorem}~\cite{S1}
Let $q$(resp. $q'$) and $p$ (resp. $p'$) be odd coprime integers such that $p>1$ (resp. $p'>1$) and
 $-p<q<p$ (resp. $-p'<q'<p'$). 
\begin{enumerate}
\item The two-bridge knots $K_{q/p}$ and $K_{q'/p'}$ are equivalent if and only if $p'=p$ and $q'=q^{\pm 1}$ (mod $p$).
\end{enumerate}
\end{theorem}

Recall that $q/p$ of $K_{q/p}$ is called \emph{slope}.
 Figure ~\ref{fig:s} shows the two-bridge knot with slope $5/9$. The diagram in Figure ~\ref{fig:s} can tighten and become the diagram of the right side in Figure ~\ref{fig:knot0}. 

 Figure~\ref{fig:knot0} shows the two equivalent knots.
 
\begin{figure} 
\begin{center}
\resizebox{8cm}{!}{\includegraphics{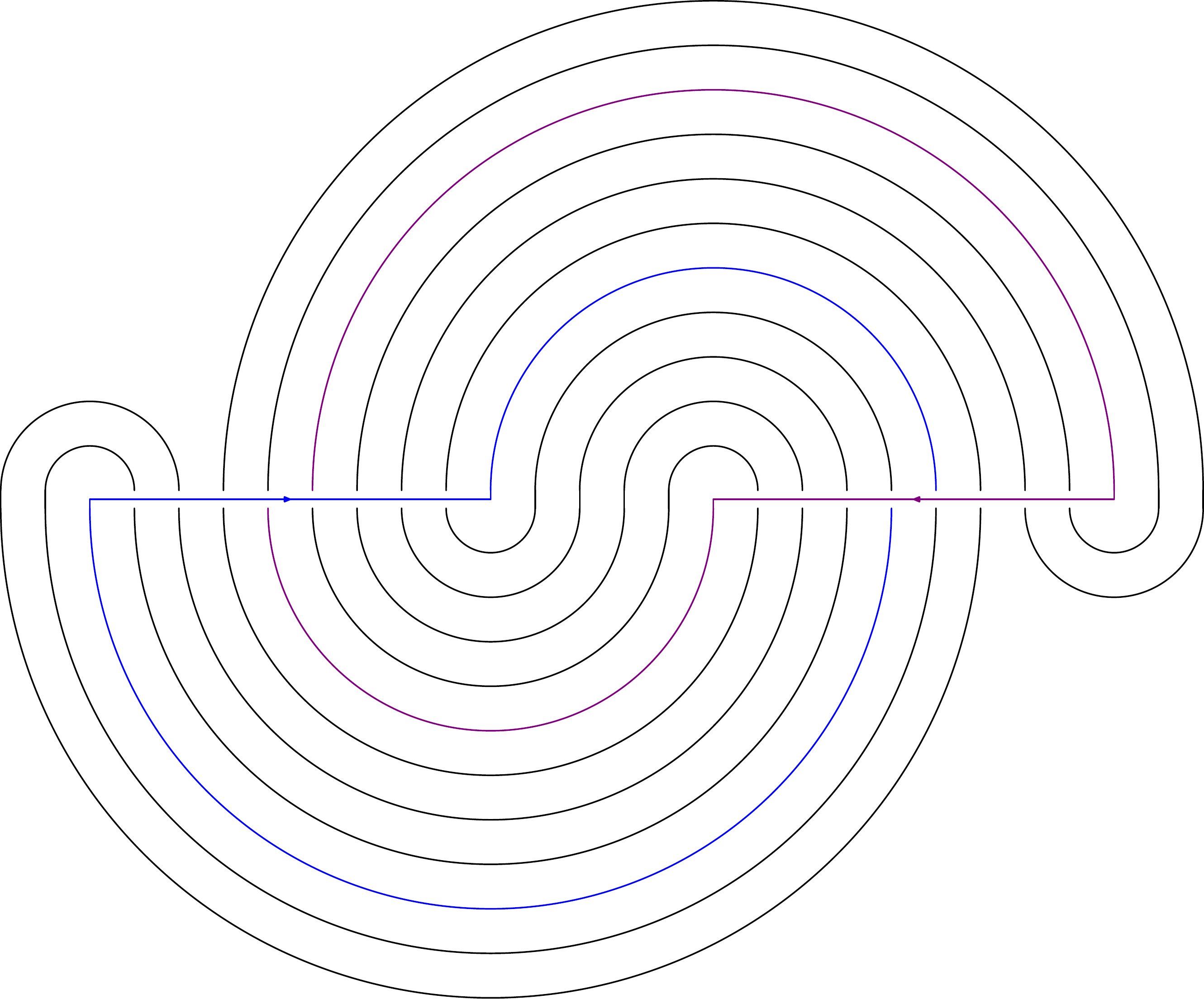}}
\caption{The knot $C(4,2)$ ($6_1$ in the Rolfsen's knot table).}
\label{fig:s}
\end{center} 
\end{figure}

\begin{figure} 
\resizebox{4.5cm}{!}{\includegraphics{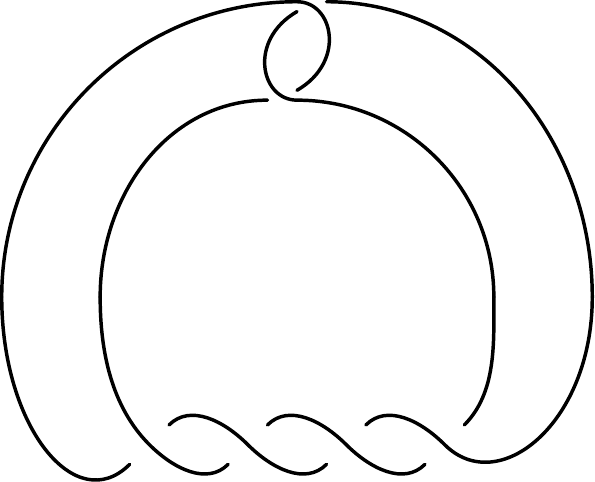}}
\ \ \
\resizebox{6cm}{!}{\includegraphics[angle=90]{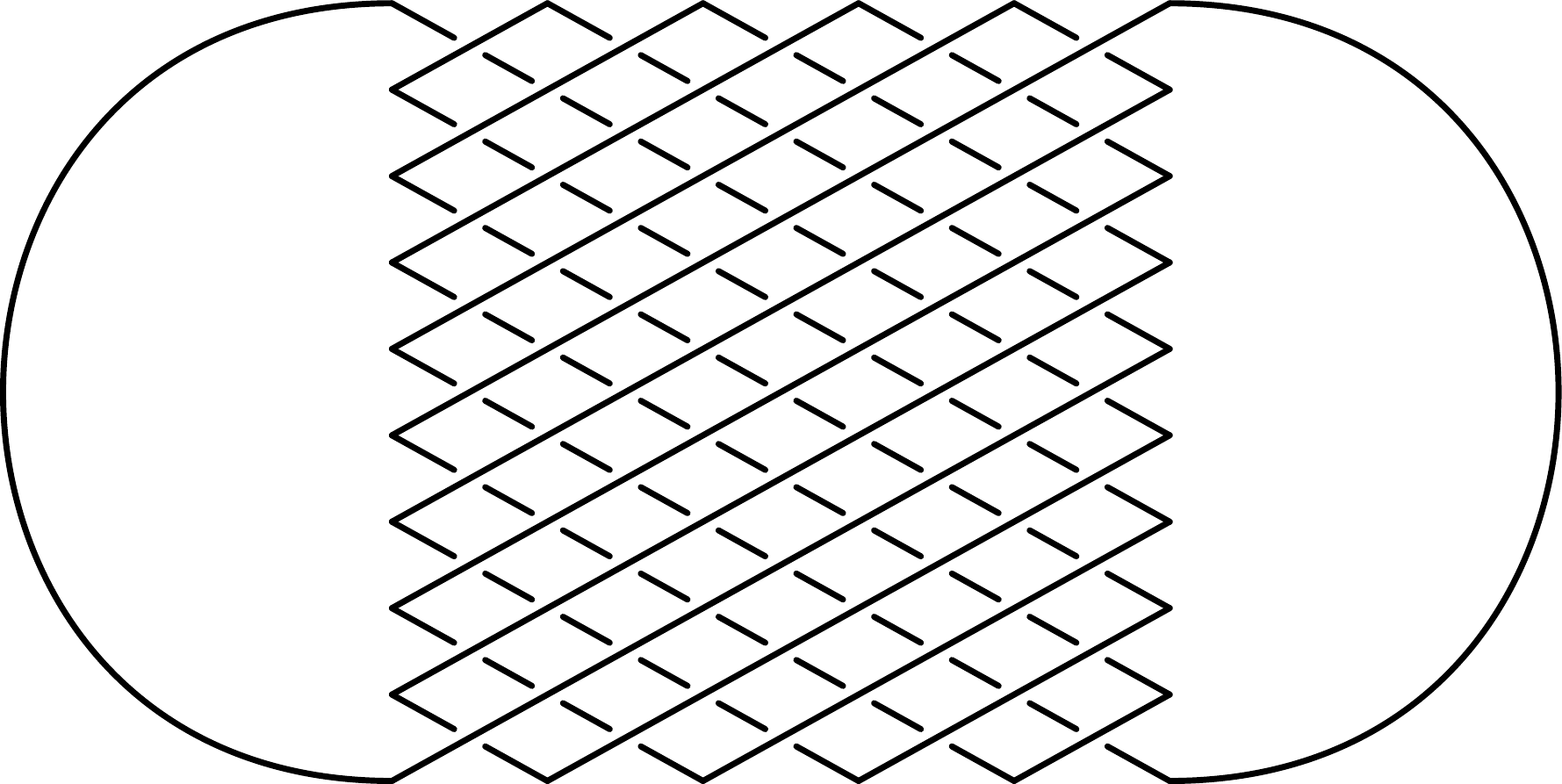}}
\caption{Knot $6_1$ with slope $2/9$ (left) and with slope $5/9$ (right).} \label{fig:knot0}
\end{figure}

\medskip

Let $K_{q/p}$ be a two bridge knot. From the Schubert normal form of $K_{q/p}$, we can read the following fundamental group~\cite{S1,R1}.

\medskip

\begin{proposition}\label{prop:fundamentalGroup}
$$\pi_1(S^3-K)=<s,t \: | \: swt^{-1}w^{-1}=1>$$
where $w=t^{\epsilon_1}s^{\epsilon_2} \cdots t^{\epsilon_{p-2}}s^{\epsilon_{p-1}}$ 
and $\epsilon_j=(-1)^{\lfloor \frac{jq}{p} \rfloor}$, for $j=1, \ldots , p-1$. 
\end{proposition}
\noindent ($\lfloor a \rfloor$ is the floor of $a$.)

\medskip

A hyperbolic (resp. Euclidean, spherical) $2$-bridge knot \emph{cone-manifold} has as its singular set the hyperbolic $2$-bridge knot $K_{q/p}$ and as its underlying space the three dimensional sphere.  Away from the $2$-bridge knot the cone-manifold is locally isometric to hyperbolic (resp. Euclidean, spherical) three dimensional space and on the $2$-bridge knot it is locally the sector of a cylinder with cone-angle $\alpha$. The special case when $\alpha$ is $2 \pi/n$ for some positive integer $n$, is called an \emph{orbifold}. We denote it by $\Oc(q/p, n)$. The topological canonical fundamental set for  
$\Oc(q/p, n)$ orbifold in the spherical space is well known~\cite{min}.

\medskip

 %%%%%%%%%%%%%%%%%%%%%%%%%%%%%%%%%%%%%%%%%%%%%%%%%%%%%%%%
%%%%%%%%%%%%%%%%%%%%%%%%%%%%%%%%%%%%%%%%%%%%%%%%%%%%%%%%
\section{The canonical fundamental set for the orbifold $\Oc(q/p,2)$ in the spherical space.}\label{sec:spherical}
 Note that there exists an angle $\alpha_0 \in [\frac{2\pi}{3},\pi)$ for each hyperbolic $K_{q/p}$ such that the cone-manifold $X_{q/p}(\alpha)$ is hyperbolic for $\alpha \in (0, \alpha_0)$, Euclidean for $\alpha=\alpha_0$, and spherical for $\alpha \in (\alpha_0, \pi]$ \cite{P2,HLM1,K1,PW}. Hence, there is only one spherical orbifold $\Oc(q/p,2)$ for each $K_{q/p}$. The $2$-fold covering of $S^3$ branched along $K_{q/p}$ is the lens space $L(p,q)$~\cite{S1} which is spherical.
The fundamental domain of $L(p,q)$ is described in~\cite[p.237-p.238]{Rol}. The strategy in this section is that constructing the half of the fundamental domain of $L(p,q)$ such that the $2$-fold covering of half of the fundamental domain of $L(p,q)$ branched along $K_{q/p}$ becomes the fundamental domain of $L(p,q)$.
Figure~\ref{fig:polygon} shows the fundamental set for the orbifold $\Oc(3/5,2)$ in the spherical space.

\begin{figure}
\begin{center}
\resizebox{7cm}{!}{\includegraphics{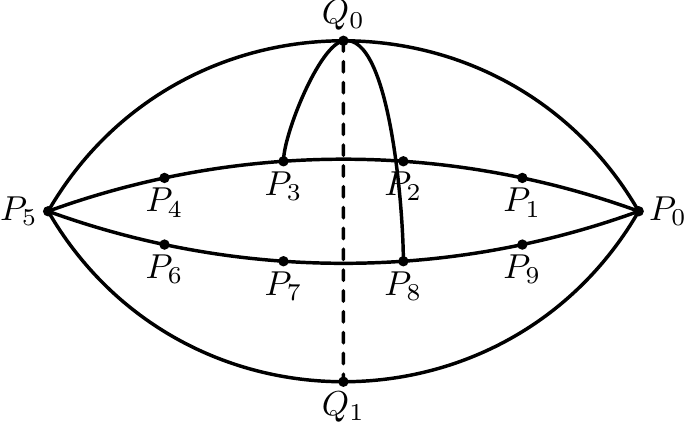}}
\caption{The fundamental set for the orbifold $\Oc(3/5,2)$.}\label{fig:polygon}
\end{center}
\end{figure}

For $K_{q/p}$, we construct a same shape set and let the dihedral angle of between the top lens and the bottom lens be $\pi/p$. Choose $2p$
 points on the circle of intersection of two lenses such that
the spherical distance between two consecutive  points can be $\pi/p$. $K_{q/p}$ is the union of the two lines passing through  $\{P_0,Q_1,P_{p}\}$ and 
$\{P_q,Q_0,P_{p+q}\}$. Let $S$ and $T$ be the half-turns along the two lines passing through  $\{P_0,Q_1,P_{p}\}$ and $\{P_q,Q_0,P_{p+q}\}$, respectively, then 
the group $\langle S,T\rangle$ is a discrete subgroup of the isometry group of $S^3$ and the fundamental set of the group $\langle S,T\rangle$ becomes the
set we constructed. For more details see~\cite{MR2}.

%%%%%%%%%%%%%%%%%%%%%%%%%%%%%%%%%%%%%%%%%%%%%%%%%%%%%%
\section{The fundamental set for the Euclidean cone-manifold $X_{q/p}(\alpha)$.} \label{sec:Euclidean}

Let $K_{q/p}$ be the hyperbolic two bridge knot ${q/p}$ and $X_{q/p}$ be the complement of $K_{q/p}$. The fundamental set for the Euclidean orbifold $\Oc(2/5,3)$ is in Figure~\ref{fig:2o5}. The fundamental set for the Euclidean orbifold $\Oc(2/5,3)$ is carefully described in~\cite{MR1}. Note that $\Oc(2/5,3)$ is equivalent to the normal form $\Oc(3/5,3)$. As we mentioned, there exists an angle $\alpha_0 \in [\frac{2\pi}{3},\pi)$ for each $K_{q/p}$ such that the cone-manifold $X_{q/p}(\alpha)$ is hyperbolic for $\alpha \in (0, \alpha_0)$, Euclidean for $\alpha=\alpha_0$, and spherical for $\alpha \in (\alpha_0, \pi]$ \cite{P2,HLM1,K1,PW}. In general, $2 \pi/\alpha$ is not an integer.
For the intermediate angles whose multiples are not $2 \pi$ and not bigger than $\pi$, you can consult a link example in~\cite{Shm}.

\begin{figure} 
\begin{center}
\resizebox{10cm}{!}{\includegraphics{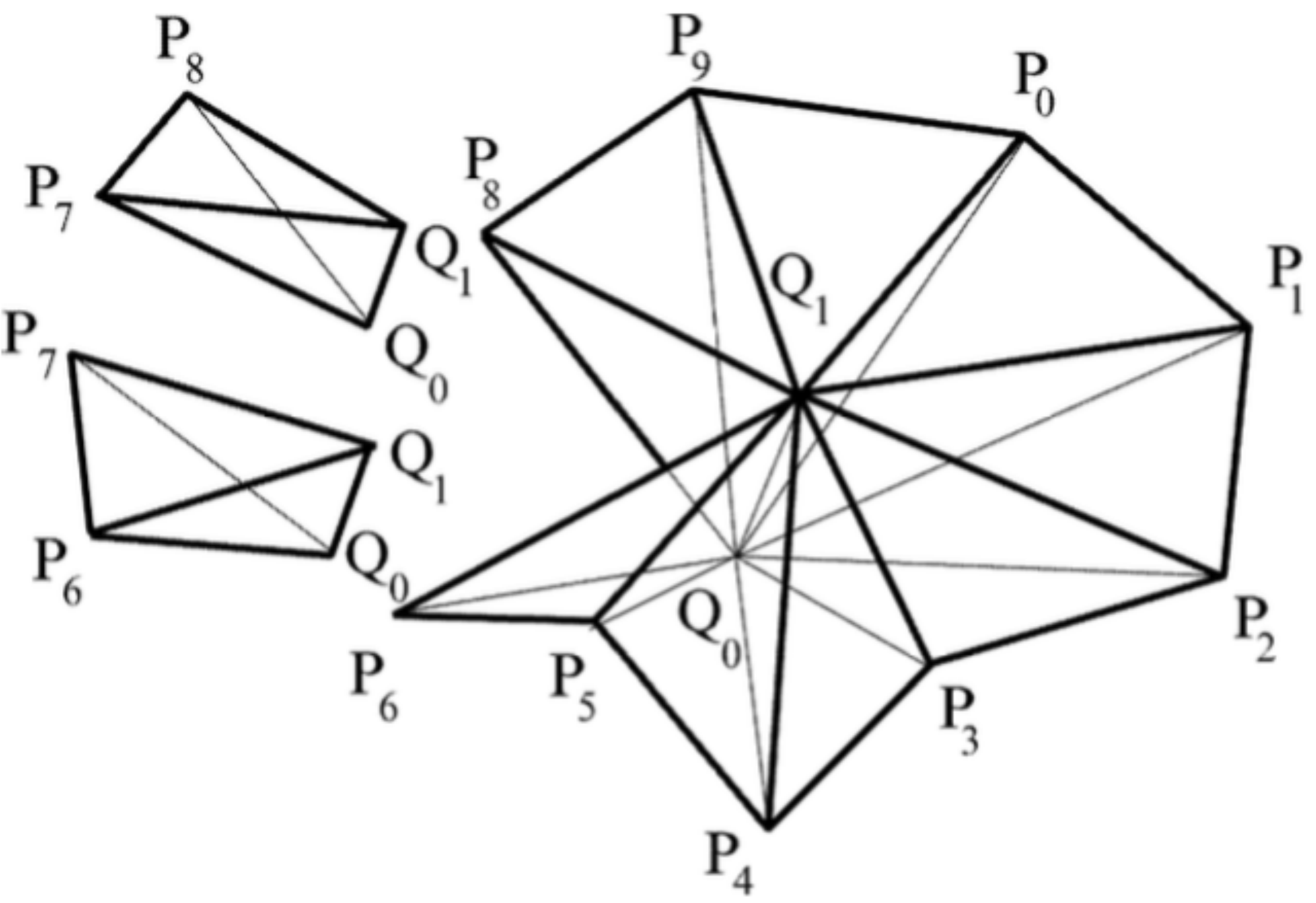}}
\caption{The knot $C(2,2)$ ($4_1$ in the Rolfsen's knot table)}
\end{center} \label{fig:2o5}
\end{figure}

%%%%%%%%%%%%%%%%%%%%%%%%%%%%%%%%%%%%%%%%%%%%%%%%%%%%%%%%
\section{Algorithm for producing the fundamental set of orbifold $\Oc(q/p,n)$ in the hyperbolic space}\label{sec:hyperbolic}
Thurston's orbifold theorem guarantees an orbifold, $\Oc(q/p,2 \pi/\alpha)$, with $K_{q/p}$ as the singular locus and the cone-angle 
$\alpha=2 \pi/k$ for some nonzero integer $k$, can be identified with 
$\Hy^3/\Gamma(n)$ for some $\Gamma(n) \in \text{PSL}(2,C)$; the hyperbolic structure of $X_{q/p}$ is deformed to the hyperbolic structure of 
$X_{q/p}(2 \pi/\alpha)$. Let $e \in \text{PSL(2,C)}$ be the element such that $eS=S^{-1}e$, $eT=T^{-1}e$, and $e^2=I$ and 
$c \in \text{PSL(2,C)}$ be the element such that $cS=T^{-1}c$, and $c^2=I$. For the $\Gamma(n)$ of orbifold $\Oc(q/p,n)$, 
we set $\Gamma(n)=\langle S,T\: |\: S^n=I,T^n=I, (SWc)^2=I \rangle$.
For the fundamental set of $\Gamma(n)$, set the unique fixed point of  $\langle S,SWc \rangle$ to be $P_0$ and the unique fixed point of 
$\langle eSe^{-1},eSWce^{-1} \rangle$ to be $P_{p}$. Then the other points of the fundamental set whose topological type is the same as that of $\Oc(q/p,2)$, can be assigned automatically. For example, $P_{2q}$ is the unique fixed point of $T \langle S,SWc \rangle T^{-1}=\langle TST^{-1},TSWcT^{-1} \rangle$. Now we connect the points with hyperbolic geodesics so that it has the topological type of the fundamental set of $\Oc(q/p,2)$. For more details see~\cite{MR2}.
%%%%%%%%%%%%%%%%%%%%%%%%%%%%%%%%%%%%%%%%%%%%%%%%%%%%%%%%%%%%%
%%%%%%%%%%%%%%%%%%%%%%%%%%%%%%%%%%%%%%%%%%%%%%%%%%%%%%%%%%%%%

\section{$\left(PSL(2,C),\mathbb{H}^3 \right) \text{ structure of } X_{q/p}(\alpha)$} \label{sec:poly}
Let $R=\text{Hom}(\pi_1(X_{q/p}), \text{SL}(2, \Cb))$. 
Given  a set of generators, $s,t$, of the fundamental group for 
$\pi_1(X_{q/p})$, we define
 a set $R\left(\pi_1(X_{q/p})\right) \subset \text{SL}(2, \Cb)^2 \subset \Cb^{8}$ to be the set of
 all points $(\rho(s),\rho(t))$, where $\rho$ is a
 representaion of $\pi_1(X_{q/p})$ into $\text{SL}(2, \Cb)$. Since the defining relation of 
 $\pi_1(X_{q/p})$ gives the defining equation of $R\left(\pi_1(X_{q/p})\right)$~\cite{R3}, $R\left(\pi_1(X_{q/p})\right)$ is an affine algebraic set in $\Cb^{8}$. 
$R\left(\pi_1(X_{q/p})\right)$ is well-defined up to isomorphisms which arise from changing the set of generators. We say elements in $R$ which differ by conjugations in $\text{SL}(2, \Cb)$ are \emph{equivalent}. 
A point on the variety gives the $\left(PSL(2,C),\mathbb{H}^3 \right) \text{ structure of } X_{q/p}(\alpha)$. Let $\rho \in R$.

Let
\begin{center}
$$\begin{array}{ccccc}
\rho(s)= \left[\begin{array}{cc}
     \cos \frac{ \alpha}{2} & i e^{\frac{d}{2}} \sin \frac{ \alpha}{2}        \\
      i e^{-\frac{d}{2}} \sin \frac{ \alpha}{2}  &  \cos \frac{ \alpha}{2}  
                     \end{array} \right]                     
\text{,} \ \ \
\rho(t)=\left[\begin{array}{cc}
        \cos \frac{ \alpha}{2} & i e^{-\frac{d}{2}} \sin \frac{ \alpha}{2}      \\
         i e^{\frac{d}{2}} \sin \frac{ \alpha}{2}   &  \cos \frac{ \alpha}{2} 
                 \end{array}  \right].
\end{array}$$
\end{center}

Then $\rho$ becomes an irreducible representation if and only if $A=\cot{\frac{\alpha}{2}}$ and $V=\cosh{d}$ satisfies a polynomial equation~\cite{R3,MR2}. 
We call the defining polynomial 
of the algebraic set $\{(V,A)\}$ as the \emph{Riley-Mednykh polynomial} for the $K_{q/p}$.
%%%%%%%%%%%%%%%%%%%%%%%%%%%%%%%%%%%%%%%%%%%%%%

\subsection{The Riley-Mednykh polynomial}\label{subsec:RM}
Given the fundamental group of $X_{q/p}$,
$$\pi_1(X_{q/p})=\left \langle s,t \ |\  swt^{-1}w^{-1}=1 \right \rangle,$$
where $w=t^{\epsilon_1}s^{\epsilon_2} \cdots t^{\epsilon_{p-2}}s^{\epsilon_{p-1}}$ 
and $\epsilon_j=(-1)^{\lfloor \frac{jq}{p} \rfloor}$, for $j=1, \ldots , p-1$. 
Let $S=\rho(s),\  T=\rho(t)$ and $W=\rho(w)$. Then the trace of $S$ and the trace of $T$ are both 
$2 \cos \frac{\alpha}{2}$. 

\begin{lemma}\label{lem:swn}
For $c \in \text{\textnormal{SL}}(2, \Cb)$ which satisfies $cS=T^{-1}c$, and $c^2=-I$,
$$SWT^{-1}W^{-1}=-(SWc)^2. $$
\end{lemma}

\begin{proof}
 \begin{equation*}
\begin{split}
 (SWc)^2 & =SWcSWc=SWT^{-1}c(T^{\epsilon_1}S^{\epsilon_2}\cdots T^{\epsilon_{p-2}}S^{\epsilon_{p-1}})c \\
              & =SWT^{-1}(S^{-\epsilon_1}T^{-\epsilon_2}\cdots S^{-\epsilon_{p-2}}T^{-\epsilon_{p-1}})c^2 \\
              & =SWT^{-1}(S^{-\epsilon_{p-1}}T^{-\epsilon_{p-2}}\cdots S^{-\epsilon_{2}}T^{-\epsilon_{1}})c^2 \\
              &=-SWS^{-1}W^{-1},
 \end{split}
\end{equation*}   
where the fourth equality comes from $\epsilon_i=\epsilon_{p-i}$~\cite{R1}.
\end{proof}

From the structure of the algebraic set of $R\left(\pi_1(X_{q/p})\right)$ with coordinates $\rho(s)$ and $\rho(t)$ we have the defining equation of 
$R\left(\pi_1(X_{q/p})\right)$. 

\begin{definition}
We define \emph{Riley-Mednykh polynomial} as $\text{\textnormal{tr}}(SWc)/\text{\textnormal{tr}}(Sc)$ up to powers of $\sin{\frac{\alpha}{2}}$. 
\end{definition}

\begin{theorem} ~\cite{MR2} \label{thm:RMpolynomial}
$\rho$ is a representation of $\pi_1(X_{q/p})$ if and only if $V$ is a root of Riley-Mednykh polynomial $P=P(V,A)$. 
\end{theorem}

\begin{proof}
Note that $SWS^{-1}W^{-1}=I$, which gives the defining equations of 
$R\left(\pi_1(X_{q/p})\right)$, is equivalent to $(SWc)^2=-I$ in $\text{SL}(2,\Cb)$ 
by Lemma~\ref{lem:swn} and  
$(SWc)^2=-I$ in $\text{SL}(2,\Cb)$ is equivalent to $\text{\textnormal{tr}}(SWc)=0$.

We can find  two $c$'s in $\text{\textnormal{SL}}(2, \Cb)$ which satisfies $cS=T^{-1}c$ and $c^2=-I$ by direct computations. The existence and the uniqueness of the isometry (the involution) which is represented by $c$ are shown in~\cite[p.46]{F}. Since two $c$'s give the same element in 
$\text{\textnormal{PSL}}(2, \Cb)$, we use one of them.
Hence, we may assume
 \begin{center}
$$\begin{array}{cc}
c=\left[\begin{array}{cc}
        0 & -1    \\
        1 & 0
       \end{array}  \right],
\end{array}$$
\end{center}

 \begin{center}
$$\begin{array}{ccccc}
S=\left[\begin{array}{cc}
     \cos \frac{ \alpha}{2} & i e^{\frac{d}{2}} \sin \frac{ \alpha}{2}        \\
      i e^{-\frac{d}{2}} \sin \frac{ \alpha}{2}  &  \cos \frac{ \alpha}{2}  
                     \end{array} \right],                          
\ \ \
T=\left[\begin{array}{cc}
        \cos \frac{ \alpha}{2} & i e^{-\frac{d}{2}} \sin \frac{ \alpha}{2}      \\
         i e^{\frac{d}{2}} \sin \frac{ \alpha}{2}   &  \cos \frac{ \alpha}{2} 
                 \end{array}  \right].          
\end{array}$$
\end{center}

Since $P$ is the defining polynomial of the algebraic set $\{(V,A)\}$ and the defining polynomial of $R\left(\pi_1(X_{q/p})\right)$ corresponding to our choice of 
$\rho(s)$ and $\rho(t)$, $P$ becomes $\text{\textnormal{tr}}(SWc)/\text{\textnormal{tr}}(Sc)$ up to powers of $\sin{\frac{\alpha}{2}}$.
\end{proof}

%%%%%%%%%%%%%%%%%%%%%%%%%%%%%%%%%%%%%%%%%%%%%%%%%%%%%%%%%%%%%
%%%%%%%%%%%%%%%%%%%%%%%%%%%%%%%%%%%%%%%%%%%%%%%%%%%%%%%%%%%%%
 
\section{Two bridge knots with  Conway's notation $C(2n, 4)$} \label{sec:C[2n,4]}

\begin{figure} 
\begin{center}
\resizebox{5cm}{!}{\includegraphics{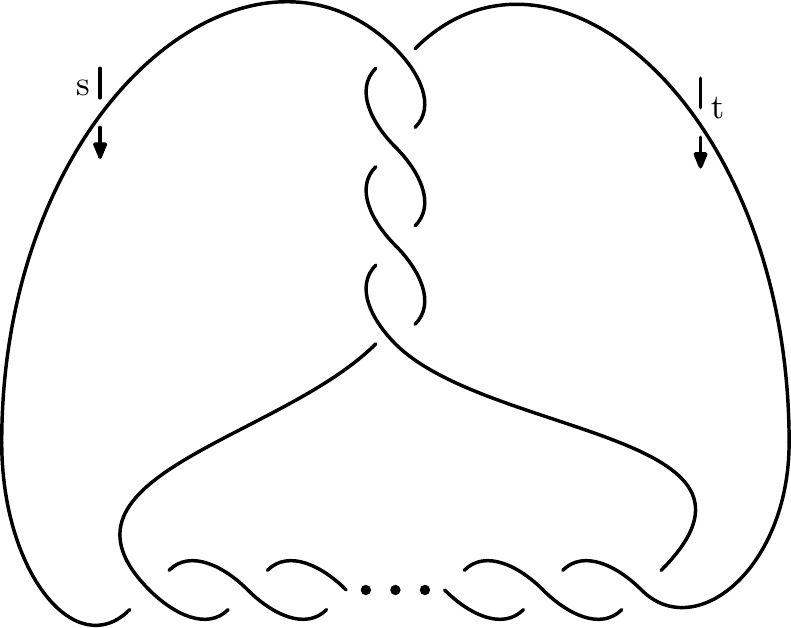}}
\qquad \qquad
\resizebox{5cm}{!}{\includegraphics{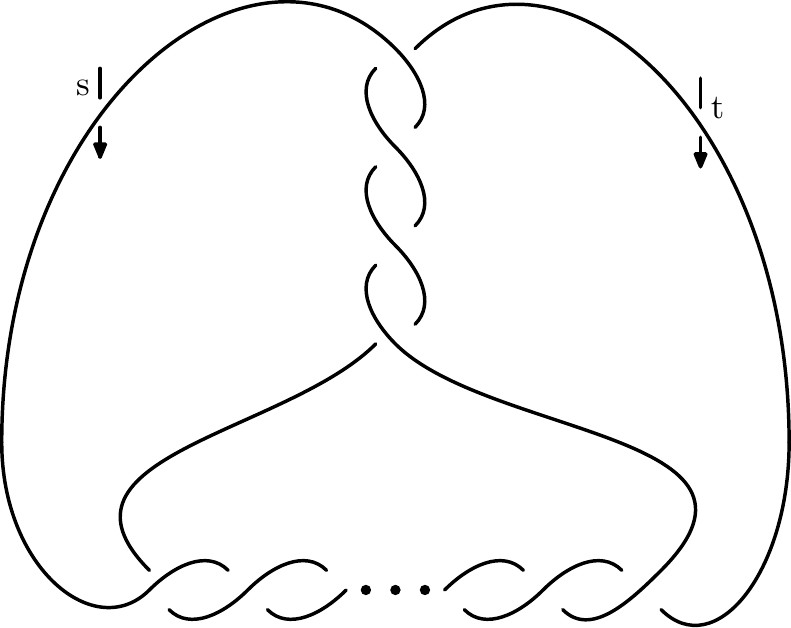}}
\caption{A two bridge knot with Conway's notation $C(2n,4)$  for $n>0$ (left) and for $n<0$ (right)} \label{fig:C[2n,4]}.
\end{center}
\end{figure}

\begin{figure}
\begin{center}
\resizebox{5cm}{!}{\includegraphics{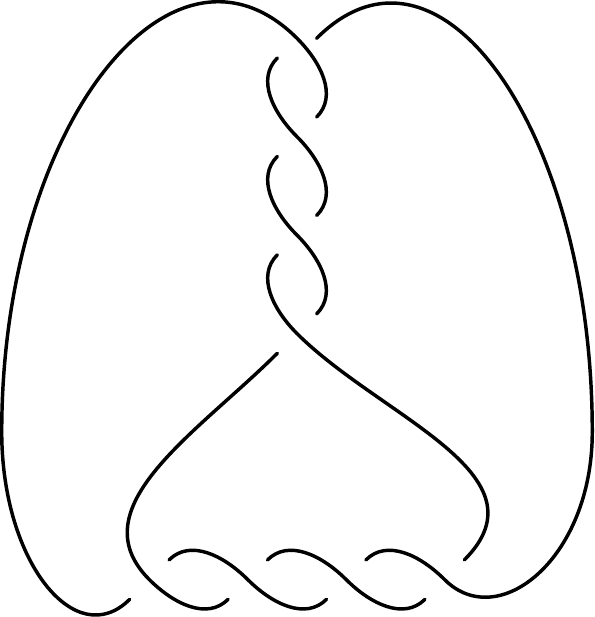}}
\caption{The knot $C(4,4)$ ($8_3$ in the Rolfsen's knot table).}\label{fig:knot}
\end{center}
\end{figure}

A knot is a two bridge knot with Conway's notation $C(2n,4)$ if it has a regular two-dimensional projection of the form in Figure~\ref{fig:C[2n,4]}. For example, Figure~\ref{fig:knot} is the knot $C(2\times 2,4)=C(4,4)$.
It has 4 left-handed vertical crossings and $2 n$ right-handed horizontal crossings ($n$ right-handed horizontal full twists). 
We denote it by $T_{2 n}$. In the rest of the paper, we will prove the following two theorems and the two following corollaries are immediate consequences.

\begin{theorem}\label{thm:main}
Let $X_{2n}(\alpha)$, $0 \leq \alpha < \alpha_0$ be the hyperbolic cone-manifold with underlying space $S^3$ and with singular set $T_{2n}$ of cone-angle 
$\alpha$. Then the volume of $X_{2n}(\alpha)$ is given by the following formula
\begin{align*}
\text{\textnormal{Vol}} \left(X_{2n}(\alpha)\right) 
&= \int_{\alpha}^{\pi} \log \left|\frac{M^{-4}-M^{-2}+(2 M^{-2}+M^{2}-1)x+x^2}{M^{4}-M^{2}+(M^{-2}+2M^{2}-1)x+x^2}\right| \: d\alpha,
\end{align*}

\noindent where 
$x$ with $\text{\textnormal{Im}}(x) \leq 0$ is a zero of the Riley-Mednykh polynomial $P_{2n}=P_{2n}(x,M)$ which is given recursively by 
\medskip
\begin{equation*}
P_{2n} = \begin{cases}
 Q P_{2(n-1)} -M^{12} P_{2(n-2)} \ \text{if $n>1$}, \\
 Q P_{2(n+1)}-M^{12} P_{2(n+2)} \ \text{if $n<-1$},
\end{cases}
\end{equation*}

\medskip
\noindent with initial conditions
\begin{equation*}
\begin{split}
P_{-2}  & =-M^4 x^3-\left(2 M^6-M^4+2 M^2\right) x^2-\left(M^8-M^6+2
   M^4-M^2+1\right) x+M^4,\\
P_{0}  & =M^{-2} \ \text{\textnormal{for}} \ n < 0 \qquad \text{\textnormal{and}} \qquad P_{0} (x,M)  =1 \ \text{\textnormal{for}} \ n > 0,\\    
P_{2}  & =M^6 x^4+\left(3 M^8-M^6+3 M^4\right) x^3+\left(3 M^{10}-2 M^8+5 M^6-2 M^4+3
   M^2\right) x^2\\
   &+\left(M^{12}-M^{10}+2 M^8-2 M^6+2 M^4-M^2+1\right) x+M^6, \\
\end{split}
\end{equation*}
\noindent and $M=e^{\frac{i \alpha}{2}}$ and

\begin{align*}
Q&=M^6 x^4+\left(3 M^8-2 M^6+3 M^4\right) x^3+\left(3 M^{10}-4 M^8+6 M^6-4 M^4+3
   M^2\right) x^2\\
   &+\left(M^{12}-2 M^{10}+3 M^8-4 M^6+3 M^4-2 M^2+1\right) x+2 M^6.
\end{align*}
\end{theorem}

The proof of Theorem~\ref{thm:main} is in Subsection~\ref{sec:proof}. From Theorem~\ref{thm:main}, the following corollary can be obtained. The following corollary gives the hyperbolic volume of the $k$-fold cyclic covering over $T_{2n}$, $M_k (X_{2n})$, for $k \geq 3$.
\begin{corollary}
The volume of $M_k (X_{2n})$ is given by the following formula

\begin{align*}
\text{\textnormal{Vol}} \left(M_k (X_{2n}))\right) &=k \, \left[\text{\textnormal{Vol}} \left(X_{2n}\left( \frac{2 \pi}{k}\right)\right)\right].
\end{align*}
\end{corollary}

\begin{theorem}\label{thm:main1}
Let $X_{2n}(\alpha)$, $0 \leq \alpha < \alpha_0$ be the hyperbolic cone-manifold with underlying space $S^3$ and with singular set $T_{2n}$ of cone-angle 
$\alpha$. Let $k$ be a positive integer such that $k$-fold cyclic covering of $X_{2n}(\frac{2 \pi}{k})$ is hyperbolic. Then the Chern-Simons invariant of 
$X_{2n}(\frac{2 \pi}{k})$ (mod $\frac{1}{k}$ if $k$ is even or mod $\frac{1}{2k}$ if $k$ is odd) is given by the following formula:

\begin{align*}
&\text{\textnormal{cs}} \left(X_{2n} \left(\frac{2 \pi}{k} \right)\right) 
 \equiv \frac{1}{2} \text{\textnormal{cs}}\left(L(8n+1,6n+1) \right) \\
&+\frac{1}{4 \pi^2}\int_{\frac{2 \pi}{k}}^{\alpha_0} Im \left(2*\log \left(-M^{-2}\frac{M^{-4}-M^{-2}+(2 M^{-2}+M^{2}-1)x+x^2}{M^{4}-M^{2}+(M^{-2}+2M^{2}-1)x+x^2}\right)\right) \: d\alpha \\
& +\frac{1}{4 \pi^2}\int_{\alpha_0}^{\pi}
 Im \left(\log \left(-M^{-2}\frac{M^{-4}-M^{-2}+(2 M^{-2}+M^{2}-1)x_1+x_1^2}{M^{4}-M^{2}+(M^{-2}+2M^{2}-1)x_1+x_1^2}\right) \right) \: d\alpha\\
 &+\frac{1}{4 \pi^2}\int_{\alpha_0}^{\pi} Im \left( \log \left(-M^{-2}\frac{M^{-4}-M^{-2}+(2 M^{-2}+M^{2}-1)x_2+x_2^2}{M^{4}-M^{2}+(M^{-2}+2M^{2}-1)x_2+x_2^2}\right)\right) \: d\alpha,
\end{align*}

\noindent where 
for $M=e^{\frac{i \alpha}{2}}$, $x$ $(Im(x) \leq 0)$, $x_1$, and $x_2$ are zeroes of Riley-Mednykh polynomial $P_{2n}=P_{2n}(x,M)$. 
For $M=e^{\frac{i \alpha}{2}}$ and $x_1$ and $x_2$ approach common $x$ 
as $\alpha$ decreases to $\alpha_0$ and they come from the components of $x$ and $\overline{x}$.
\end{theorem}

The proof of Theorem~\ref{thm:main1} is in Subsection~\ref{sec:proof1}. From Theorem~\ref{thm:main1}, the following corollary can be obtained. The following corollary gives the Chern-Simons invariant of the $k$-fold cyclic covering over $T_{2n}$, $M_k (X_{2n})$, for $k \geq 3$.
\begin{corollary}
The Chern-Simons invariant of $M_k (X_{2n})$ is given by the following formula
\begin{align*}
\text{\textnormal{cs}} \left(M_k (X_{2n}))\right) &=k \, \left[\text{\textnormal{cs}} \left(X_{2n}\left(\frac{2 \pi}{k}\right)\right)\right].
\end{align*}
\end{corollary}

In~\cite[Proposition 1]{R1}, the fundamental group of two-bridge knots is presented. We will use the fundamental group of $X_{2n}$ in~\cite{HS}. In~\cite{HS}, the fundamental group of $X_{2n}$ is calculated with 4 right-handed vertical crossings as positive crossings instead of 4 left-handed vertical crossings. The following proposition is tailored to our purpose. The reduced word relation of the one in the following proposition can also be obtained by reading off the fundamental group from the Schubert normal form of $T_{2n}$ with slope $\frac{6n+1}{8n+1}$~\cite{S1,R1}.

\begin{proposition}\label{thm:fundamentalGroup}
$$\pi_1(X_{2n})=\left\langle s,t \ |\ swt^{-1}w^{-1}=1\right\rangle,$$
where $w=(ts^{-1}ts^{-1}t^{-1}st^{-1}s)^n$.
\end{proposition}
%%%%%%%%%%%%%%%%%%%%%%%%%%%%%%%%%%%%%%%%%%%%%%%%%%%%%%%%%%%%%%%
\subsection{The Riley-Mednykh polynomial}
Since we are interested in the excellent component (the geometric component) of $R\left(\pi_1(X_{2n})\right)$, 
in this subsection we set 
$M=e^{\frac{i \alpha}{2}}$.
Given the fundamental group of $X_{2n}$,
$$\pi_1(X_{2n})=\left \langle s,t \ |\  swt^{-1}w^{-1}=1 \right \rangle,$$
where $w=(ts^{-1}ts^{-1}t^{-1}st^{-1}s)^n$, let $S=\rho(s),\  T=\rho(t)$ and $W=\rho(w)$ for $\rho \in R\left(\pi_1(X_{2n})\right)$. Then the trace of $S$ and the trace of $T$ are both 
$2 \cos \frac{\alpha}{2}$. 

\begin{lemma}\label{lem:swc}
For $c \in \text{\textnormal{SL}}(2, \C)$ which satisfies $cS=T^{-1}c$ and $c^2=-I$,
$$SWT^{-1}W^{-1}=-(SWc)^2 $$
\end{lemma}
\begin{proof}
 \begin{equation*}
\begin{split}
 (SWc)^2 & =SWcSWc=SWT^{-1}c(TS^{-1}TS^{-1}T^{-1}ST^{-1}S)^nc \\
              & =SWT^{-1}(S^{-1}TS^{-1}TST^{-1}ST^{-1})^nc^2=-SWT^{-1}W^{-1}.
 \end{split}
\end{equation*}   
\end{proof}

From the structure of the algebraic set of $R\left(\pi_1(X_{2n})\right)$ with coordinates $\rho(s)$ and $\rho(t)$ we have the defining equation of 
$R\left(\pi_1(X_{2n})\right)$. By plugging in $e^{\frac{i \alpha}{2}}$ into $M$ of that equation, we have the following theorem.

\begin{theorem} \label{thm:RMpolynomial}
$\rho$ is a representation of $\pi_1(X_{2n})$ if and only if $x$ is a root of the following Riley-Mednykh polynomial $P_{2n}=P_{2n}(x,M)$ which is given recursively by 

\medskip
\begin{equation*}
P_{2n} = \begin{cases}
 Q P_{2(n-1)} -M^{12} P_{2(n-2)} \ \text{if $n>1$}, \\
 Q P_{2(n+1)}-M^{12} P_{2(n+2)} \ \text{if $n<-1$},
\end{cases}
\end{equation*}
\medskip

\noindent with initial conditions
\begin{equation*} 
\begin{split}
P_{-2}  & =-M^4 x^3-\left(2 M^6-M^4+2 M^2\right) x^2-\left(M^8-M^6+2
   M^4-M^2+1\right) x+M^4,\\
P_{0}  & =M^{-2} \ \text{\textnormal{for}} \ n < 0 \qquad \text{\textnormal{and}} \qquad P_{0} (x,M)  =1 \ \text{\textnormal{for}} \ n > 0,\\    
P_{2}  & =M^6 x^4+\left(3 M^8-M^6+3 M^4\right) x^3+\left(3 M^{10}-2 M^8+5 M^6-2 M^4+3
   M^2\right) x^2\\
   &+\left(M^{12}-M^{10}+2 M^8-2 M^6+2 M^4-M^2+1\right) x+M^6, \\
\end{split}
\end{equation*}
\noindent and $M=e^{\frac{i \alpha}{2}}$ and
\begin{align*}
Q&=M^6 x^4+\left(3 M^8-2 M^6+3 M^4\right) x^3+\left(3 M^{10}-4 M^8+6 M^6-4 M^4+3
   M^2\right) x^2\\
   &+\left(M^{12}-2 M^{10}+3 M^8-4 M^6+3 M^4-2 M^2+1\right) x+2 M^6.
\end{align*}
\end{theorem}

\begin{proof}
Note that $SWT^{-1}W^{-1}=I$, which gives the defining equations of 
$R\left(\pi_1(X_{2n})\right)$, is equivalent to $(SWc)^2=-I$ in $\text{SL}(2,\C)$ 
by Lemma~\ref{lem:swc} and  
$(SWc)^2=-I$ in $\text{SL}(2,\C)$ is equivalent to $\text{\textnormal{tr}}(SWc)=0$.

We can find  two $c$'s in $\text{\textnormal{SL}}(2, \C)$ which satisfies $cS=T^{-1}c$ and $c^2=-I$ by direct computations. The existence and the uniqueness of the isometry (the involution) which is represented by $c$ are shown in~\cite[p. 46]{F}. Since two $c$'s give the same element in $\text{\textnormal{PSL}}(2, \C)$, we use one of them.
Hence, we may assume
 \begin{center}
$$\begin{array}{cc}
c=\left[\begin{array}{cc}
        0 & -\frac{M }{\sqrt{-1+2 M^2-M^4-M^2 x}}     \\
        \frac{\sqrt{-1+2 M^2-M^4-M^2 x}}{M} & 0
       \end{array}  \right],
\end{array}$$
\end{center}

 \begin{center}
$$\begin{array}{ccccc}
S=\left[\begin{array}{cc}
     e^ {\frac{ i \alpha}{2} }& 1       \\
     0 &  e^{-\frac{i  \alpha}{2}} 
                     \end{array} \right],                          
\ \ \
T=\left[\begin{array}{cc}
     e^ { \frac{i \alpha}{2} }& 0       \\
     2-e^{i \alpha}-e^{-i \alpha}-x &  e^{ -\frac{i \alpha}{2}} 
                     \end{array} \right],            
\end{array}$$
\end{center}

and let $U=TS^{-1}TS^{-1}T^{-1}ST^{-1}S$. Note that $\text{\textnormal{tr}}(S)$ and $\text{\textnormal{tr}}(T)$ are the same as $\text{\textnormal{tr}}(S)$ 
and $\text{\textnormal{tr}}(T)$ of subsection~\ref{subsec:RM}. 

Recall that $P_{2n}$ is the defining polynomial of the algebraic set $\{(M,x)\}$ and the defining polynomial of $R\left(\pi_1(X_{2n})\right)$ corresponding to our choice of $\rho(s)$ and $\rho(t)$. We will show that 
$P_{2n}$ is equal to $\text{\textnormal{tr}}(SWc) / \text{\textnormal{tr}}(Sc)$ times $M^{6n}$ for $n>0$ and $M^{6n-2}$ for $n<0$. One can easily see 
$\text{\textnormal{tr}}(SUc)=P_2 \text{\textnormal{tr}}(Sc)/M^6$, $\text{\textnormal{tr}}(SU^{-1}c)=P_{-2} \text{\textnormal{tr}}(Sc)/M^4$ and $\text{\textnormal{tr}}(U)=Q/M^6=\text{\textnormal{tr}}(U^{-1})$. We set $P_0$ as the statement of the theorem.
Now, we only need the following recurrence relations.

\begin{align*}
\text{\textnormal{tr}}(SWc) & =\text{\textnormal{tr}}(SU^nc)=\text{\textnormal{tr}}(SU^{n-1}cU^{-1})=\text{\textnormal{tr}}(SU^{n-1}c)\text{\textnormal{tr}}(U^{-1})-\text{\textnormal{tr}}(SU^{n-1}cU) \\
            & =\text{\textnormal{tr}}(SU^{n-1}c)\text{\textnormal{tr}}(U^{-1})-\text{\textnormal{tr}}(SU^{n-2}c)=0 \ \text{if $n>1$}
\end{align*}
or
\begin{equation*}
\begin{split}
\text{\textnormal{tr}}(SWc) & =\text{\textnormal{tr}}(SU^nc)=\text{\textnormal{tr}}(SU^{n+1}cU)=\text{\textnormal{tr}}(SU^{n+1}c)\text{\textnormal{tr}}(U)-\text{\textnormal{tr}}(SU^{n+1}cU^{-1}) \\
            & =\text{\textnormal{tr}}(SU^{n+1}c)\text{\textnormal{tr}}(U)-\text{\textnormal{tr}}(SU^{n+2}c)=0 \ \text{if $n<-1$},
\end{split}
\end{equation*}
where the third equality comes from the Cayley-Hamilton theorem. Since
$\text{\textnormal{tr}}(Sc)$, $\text{\textnormal{tr}}(SUc)$, and $\text{\textnormal{tr}}(SU^{-1}c)$ have 
$\text{\textnormal{tr}}(Sc)=\frac{\sqrt{-1+2 M^2-M^4-M^2 x}}{M}$
 as a common factor, all of 
$\text{\textnormal{tr}}(SWc)$'s have $\text{\textnormal{tr}}(Sc)$ as a common factor. We left the common factor out of $\text{\textnormal{tr}}(SWc)$, multiplied it by a power of $M^{6n}$ if $n>0$ and $M^{6n-2}$ for $n<0$ to clear the fractions and denote it by $P_{2n}$, the Riley-Mednykh polynomial.
\end{proof}

%%%%%%%%%%%%%%%%%%%%%%%%%%%%%%%%%%%%%%%%%%%%%%%%%%%%%%%%%%%%%
%%%%%%%%%%%%%%%%%%%%%%%%%%%%%%%%%%%%%%%%%%%%%%%%%%%%%%%%%%%%%
\subsection{Longitude}
\label{sec:longitude}
   Let $l = ww^{*}$, where $w^{*}$ is the word obtained by reversing $w$. Let $L=\rho(l)_{11}$. Then $l$ is the longitude which is null-homologus in $X_{2n}$. Recall $\rho(w)=U^n$. We can write $\rho(w^{*})=\widetilde{U}^n$. It is easy to see that $U$ and $\widetilde{U}$ can be written as
   
   $$U=\begin{pmatrix}
u_{11} & u_{12} \\
u_{21} & u_{22}
\end{pmatrix}$$
and 
$$\widetilde{U}=
\begin{pmatrix}
\tilde{u}_{22} & \tilde{u}_{12} \\
\tilde{u}_{21}& \tilde{u}_{11}
\end{pmatrix}
$$
   
where $\tilde{u}_{ij}$ is obtained by $u_{ij}$ by replacing $M$ with $M^{-1}$. Similar computation was introduced in~\cite{HS}.

\begin{definition} \label{def:longitude}
 The \emph{complex length} of the longitude $l$ is the complex number 
 $\gamma_{\alpha}$ modulo $4 \pi \Z$ satisfying 
\begin{align*}
 \text{\textnormal{tr}}(\rho(l))=2 \cosh \frac{\gamma_{\alpha}}{2}.
\end{align*}
 Note that 
 $l_{\alpha}=|Re(\gamma_{\alpha})|$ is the real length of the longitude of the cone-manifold $X_{2n}(\alpha)$.
\end{definition}

The following lemma was introduced in~\cite{HS} with slightly different coordinates. 

\begin{lemma}~\cite{HS}
\label{lem:lemma}
$u_{21} L+\tilde{u}_{21}=0$.
\end{lemma}

\begin{theorem}
\label{thm:longitude}
Recall that $L=\rho(l)_{11}$ where $l$ is the longitude of $X_{2n}$, $M=e^{\frac{i \alpha}{2}}$, and $x$ is a root of the following Riley-Mednykh polynomial $P_{2n}$. We have
\begin{align*}
L=-M^{-2}\frac{M^{-4}-M^{-2}+(2 M^{-2}+M^{2}-1)x+x^2}{M^{4}-M^{2}+(M^{-2}+2M^{2}-1)x+x^2}.
\end{align*}
\end{theorem}

\begin{proof}
By directly computing $u_{21} L+\tilde{u}_{21}=0$ in Lemma~\ref{lem:lemma}, the theorem follows.
\end{proof}

%%%%%%%%%%%%%%%%%%%%%%%%%%%%%%%%%%%%%%%%%%%%%%%%%%%%%%%%
%%%%%%%%%%%%%%%%%%%%%%%%%%%%%%%%%%%%%%%%%%%%%%%%%%%%%%%%
\subsection{Schl\"{a}fli formula for the volume}
We will use the following Schl\"{a}fli formula to prove Theorem~\ref{thm:main}.

\begin{theorem}~\cite[Theorem 3.20]{CHK}
Let $X_{2n}(\alpha)$ be a family of cone-manifold structures of constant curvature $-1$. Assume that the underlying space is $\mathbb{S}^3$ and the singular locus is the knot $T_{2n}$. Then the derivative of volume $V \left(X_{2n}(\alpha)\right)$ of $X_{2n}(\alpha)$ satisfies
\begin{align*}
-2 \frac{dV \left(X_{2n}(\alpha)\right)}{d \alpha} =l_{\alpha}.
\end{align*}
\end{theorem}
%%%%%%%%%%%%%%%%%%%%%%%%%%%%%%%%%%%%%%%%%%%%%%%%%%%%%%%%
%%%%%%%%%%%%%%%%%%%%%%%%%%%%%%%%%%%%%%%%%%%%%%%%%%%%%%%%
\subsection{Proof of Theorem~\ref{thm:main}} \label{sec:proof}

For $n \geq 1$ and $M=e^{i \frac{\alpha}{2}}$, $P_{2n}(x,M)$ have $4n$ component zeros, and for  $n \leq -1$, $-(4n+1)$ component zeros. The component which passes through $(x_1,x_2)=\left(2-2 \cos{\left(\frac{\pi(-2n+1)}{8n+1}\right)},2-2 \cos{\left(\frac{\pi(-2n-1)}{8n+1}\right)}\right)$ at $\alpha=\pi$ is the geometric component by ~\cite[Theorem 2.1]{HLM2}, $2-x_1>0$ and $2-x_2>0$. For each $n$, there exists an angle $\alpha_0 \in [\frac{2\pi}{3},\pi)$ such that $T_{2n}(\alpha)$ is hyperbolic for $\alpha \in (0, \alpha_0)$, Euclidean for $\alpha=\alpha_0$, and spherical for $\alpha \in (\alpha_0, \pi]$ \cite{P2,HLM1,K1,PW}. 
From the following Equality~(\ref{equ:absL}), when $|L|=1$, which happens when $\alpha=\alpha_0$, $\text{\textnormal{Im}}(x)=0$. Hence, when 
$\alpha$ increases from $0$ to $\alpha_0$, two complex numbers $x$ and $\overline{x}$ approach a same real number. In other words, 
$P_{2n}(x,e^{\frac{i \alpha_0}{2}})$ has a multiple root.
Denote by $D(X_{2n}(\alpha))$  the discriminant of
$P_{2n}(x,M)$ over $x$. Then $\alpha_0$ will be one of the zeros of $D(X_{2n}(\alpha))$.

From Theorem~\ref{thm:longitude}, for some nonnegative real numbers $a$ and $b$, we have the following equality,
\begin{equation}\label{equ:absL}
\begin{split}
|L|^2 &= \left|\frac{M^{-4}-M^{-2}+(2 M^{-2}+M^{2}-1)x+x^2}{M^{4}-M^{2}+(M^{-2}+2M^{2}-1)x+x^2}\right|^2
= \frac{ a-{\rm Im}(x)b}{a + {\rm Im}(x)b}.
\end{split}
\end{equation}

For the volume, we choose $L$ with $|L|\geq1$ and hence we have $\text{\textnormal{Im}}(x) \leq 0$ by Equality~(\ref{equ:absL}). 
On the geometric component we have
 the volume of a hyperbolic cone-manifold 
$X_{2n}(\alpha)$ for $0 \leq \alpha < \alpha_0$:
\begin{align*}
\text{\textrm{Vol}}(X_{2n}(\alpha)) &=-\int_{\alpha_0}^{\alpha} \frac{l_{\alpha}}{2} \: d\alpha \\
                        &=-\int_{\alpha_0}^{\alpha} \log|L| \: d\alpha\\
                         &=-\int_{\pi}^{\alpha} \log|L| \: d\alpha\\
                         &=\int^{\pi}_{\alpha} \log|L| \: d\alpha\\
                         &=\int^{\pi}_{\alpha}  \log \left|\frac{M^{-4}-M^{-2}+(2 M^{-2}+M^{2}-1)x+x^2}{M^{4}-M^{2}+(M^{-2}+2M^{2}-1)x+x^2}\right|\: d\alpha,                       
\end{align*}
where the first equality comes from the Schl\"{a}fli formula for cone-manifolds (Theorem 3.20 of~\cite{CHK}), the second equality comes from the fact that $l_{\alpha}=|Re(\gamma_{\alpha})|$ is the real length of the longitude of 
$X_{2n}(\alpha)$, the third equality comes from the fact that $\log|L|=0$  for $\alpha_0 < \alpha \leq \pi$ by Equality~\ref{equ:absL} since all the characters are real (the proof of Proposition 6.4 of~\cite{PW}) for $\alpha_0 < \alpha \leq \pi$, and 
$\alpha_0 \in [\frac{2 \pi}{3},\pi)$ is a zero of the discriminant $D(X_{2n}(\alpha))$.

\medskip

%%%%%%%%%%%%%%%%%%%%%%%%%%%%%%%%%%%%%%%%%%%%%%%%%%%%%%%%%
%%%%%%%%%%%%%%%%%%%%%%%%%%%%%%%%%%%%%%%%%%%%%%%%%%%%%%%%%
\subsection{Schl\"{a}fli formula for the generalized Chern-Simons function}
\label{sec:CSfunction}

The general references for this section are~\cite{HLM3,HLM2,Y1,MeyRub1} and~\cite{HL}. 
We introduce the generalized Chern-Simons function on the family of $C(2n,4)$ cone-manifold structures. For the oriented knot $T_{2n}$, we orient a chosen meridian $s$ such that the orientation of $s$ followed by orientation of $T_{2n}$ coincides with orientation of $S^3$. Hence, we use the definition of Lens space in~\cite{HLM2} so that we can have the right orientation when the definition of Lens space is combined with the following frame field. On the Riemannian manifold $S^3-T_{2n}-s$ we choose a special frame field $\Gamma$. A \emph{special} frame field $\Gamma=(e_1,e_2,e_3)$ is an orthonomal frame field such that for each point $x$ near $T_{2n}$, $e_1(x)$ has the knot direction, $e_2(x)$ has the tangent direction of a meridian curve, and $e_3(x)$ has the knot to point direction.  A special frame field always exists by Proposition $3.1$ of~\cite{HLM3}. From $\Gamma$ we obtain an orthonomal frame field 
$\Gamma_{\alpha}$ on $X_{2n}(\alpha)-s$ by the Schmidt orthonormalization process with respect to the geometric structure of the cone manifold $X_{2n}(\alpha)$. Moreover it can be made special by deforming it in a neighborhood of the singular set and 
$s$ if necessary. $\Gamma^{\prime}$ is an extention of $\Gamma$ to $S^3-T_{2n}$. For each cone-manifold $X_{2n}(\alpha)$, we assign the real number:

\begin{equation*}
I\left(X_{2n}(\alpha)\right)=\frac{1}{2} \int_{\Gamma(S^3-T_{2n}-s)}Q-\frac{1}{4 \pi} \tau(s,\Gamma^{\prime})-\frac{1}{4 \pi} \left(\frac{\beta \alpha}{2 \pi}\right),
\end{equation*}

\noindent where $-2 \pi \leq \beta \leq 2 \pi$, $Q$ is the Chern-Simons form:

\begin{equation*}
Q=\frac{1}{4 \pi^2} \left(\theta_{12} \wedge \theta_{13} \wedge \theta_{23} + \theta_{12} \wedge \Omega_{12} + \theta_{13} \wedge \Omega_{13} + \theta_{23} \wedge \Omega_{23} \right),
\end{equation*}

\noindent and 

\begin{equation*}
\tau(s,\Gamma^{\prime})=-\int_{\Gamma^{\prime}(s)} \theta_{23},
\end{equation*}

\noindent where ($\theta_{ij}$) is the connection $1$-form, ($\Omega_{ij}$) is the curvature $2$-form of the Riemannian connection on $X_{2n}(\alpha)$ and the integral is over the orthonomalizations of the same frame field. When $\alpha = \frac{2 \pi}{k}$ for some positive integer, 
$I \left(X_{2n}\left(\frac{2 \pi}{k}\right)\right)$ (mod $\frac{1}{k}$ if $k$ is even or mod $\frac{1}{2k}$ if $k$ is odd) is independent of the frame field $\Gamma$ and of the representative in the equivalence class $\overline{\beta}$ and hence an invariant of the orbifold $X_{2n}\left(\frac{2 \pi}{k}\right)$. $I \left(X_{2n}\left(\frac{2 \pi}{k}\right)\right)$ (mod $\frac{1}{k}$ if $k$ is even or mod $\frac{1}{2k}$ if $k$ is odd) is called \emph{the Chern-Simons invariant of the orbifold} and is denoted by 
$\text{\textnormal{cs}} \left(X_{2n}\left(\frac{2 \pi}{k}\right) \right)$.

On the generalized Chern-Simons function on the family of $C(2n,4)$ cone-manifold structures we have the following Schl\"{a}fli formula.

\begin{theorem}(Theorem 1.2 of~\cite{HLM2})~\label{theorem:schlafli}
For a family of geometric cone-manifold structures, $X_{2n}(\alpha)$, and differentiable functions $\alpha(t)$ and $\beta(t)$ of $t$ we have
\begin{equation*}
dI \left(X_{2n}(\alpha)\right)=-\frac{1}{4 \pi^2} \beta d \alpha.
\end{equation*} 
\end{theorem}
%%%%%%%%%%%%%%%%%%%%%%%%%%%%%%%%%%%%%%%%%%%%%%%%%%%%%%%%
%%%%%%%%%%%%%%%%%%%%%%%%%%%%%%%%%%%%%%%%%%%%%%%%%%%%%%%%

\subsection{Proof of the theorem~\ref{thm:main1}} \label{sec:proof1}

 On the geometric component which is identified in the proof of Theorem~\ref{thm:main}, we can calculate
 the Chern-Simons invariant of an orbifold 
$X_{2n}(\frac{2 \pi}{k})$ (mod $\frac{1}{k}$ if $k$ is even or mod $\frac{1}{2k}$ if $k$ is odd), where $k$ is a positive integer such that $k$-fold cyclic covering of $X_{2n}(\frac{2 \pi}{k})$ is hyperbolic:
\begin{align*}
& \text{\textnormal{cs}}\left(X_{2n} \left(\frac{2 \pi}{k} \right)\right) 
                        \equiv I \left(X_{2n} \left(\frac{2 \pi}{k} \right)\right) 
                      \ \ \ \ \ \ \ \ \ \ \ \  \left(\text{mod} \ \frac{1}{k}\right) \\
                        & \equiv I \left(X_{2n}( \pi) \right)
                          +\frac{1}{4 \pi^2}\int_{\frac{2 \pi}{k}}^{\pi} \beta \: d\alpha 
                      \ \ \ \ \ \ \ \ \ \ \ \ \  \ \ \ \ \ \ \  \left(\text{mod} \ \frac{1}{k}\right) \\
                        &   \equiv \frac{1}{2} \text{\textnormal{cs}}\left(L(8n+1,6n+1) \right) \\
&+\frac{1}{4 \pi^2}\int_{\frac{2 \pi}{k}}^{\alpha_0} Im \left(2*\log \left(-M^{-2}\frac{M^{-4}-M^{-2}+(2 M^{-2}+M^{2}-1)x+x^2}{M^{4}-M^{2}+(M^{-2}+2M^{2}-1)x+x^2}\right)\right) \: d\alpha \\
& +\frac{1}{4 \pi^2}\int_{\alpha_0}^{\pi}
 Im \left(\log \left(-M^{-2}\frac{M^{-4}-M^{-2}+(2 M^{-2}+M^{2}-1)x_1+x_1^2}{M^{4}-M^{2}+(M^{-2}+2M^{2}-1)x_1+x_1^2}\right) \right) \: d\alpha\\
 &+\frac{1}{4 \pi^2}\int_{\alpha_0}^{\pi} Im \left( \log \left(-M^{-2}\frac{M^{-4}-M^{-2}+(2 M^{-2}+M^{2}-1)x_2+x_2^2}{M^{4}-M^{2}+(M^{-2}+2M^{2}-1)x_2+x_2^2}\right)\right) \: d\alpha \\
& \left( \text{mod} \ \frac{1}{k}\ \text{if $k$ is even or }  \text{mod} \ \frac{1}{2k}\ \text{if $k$ is odd} \right)
\end{align*}
where the second equivalence comes from Theorem~\ref{theorem:schlafli} and the third equivalence comes from the fact that $I \left(X_{2n}(\pi)\right) \equiv \frac{1}{2} \text{\textnormal{cs}}\left(L(8n+1,6n+1) \right)$  $\left(\text{mod }\frac{1}{2}\right)$, Theorem~\ref{thm:longitude}, and geometric interpretations of hyperbolic and spherical holonomy representations.

The following theorem gives the Chern-Simons invariant of the Lens space $L(8n+1,6n+1)$.

\begin{theorem}(Theorem 1.3 of~\cite{HLM2}) \label{theorem:Lens}
\begin{align*}
\text{\textnormal{cs}} \left(L \left(8n+1,6n+1\right)\right) \equiv \frac{7n+3}{8n+1} && (\text{mod}\ 1).
\end{align*}
\end{theorem}

%%%%%%%%%%%%%%%%%%%%%%%%%%%%%%%%%%%%%%%%%%%%%%%%%%%%%%%%%%%%%%%%%%
\subsection{Numerical computations}
Table~\ref{table1} gives the numerical approximation of $\alpha_0$ and the Chern-Simons invariant of $X_{2n}$ for $n$ between $1$ and $9$ and for $n$ between $-9$ and $-1$. We used Simpson's rule for the approximation with $2 \times10^4$ ($10^4$ in Simpson's rule) intervals from $0$ to $\alpha_0$ and $2 \times 10^4$ ($10^4$ in Simpson's rule) intervals from $\alpha_0$ to $\pi$. We used higher precision than the normal in Mathematica.
Table~\ref{table2-1} (resp. Table~\ref{table2-2}) gives the approximate Chern-Simons invariant of the hyperbolic orbifold, 
$\text{\textnormal{cs}} \left(X_{2n} (\frac{2 \pi}{k})\right)$ for $n$ between $1$ and $9$ (resp. for $n$ between $-9$ and $-1$) and for $k$ between $3$ and $10$, and of its cyclic covering, $\text{\textnormal{cs}} \left(M_k (X_{2n})\right)$. We used Simpson's rule for the approximation with $2 \times10^2$ ($10^2$ in Simpson's rule) intervals from $2 \pi/k$ to $\alpha_0$ and $2 \times 10^2$ ($10^2$ in Simpson's rule) intervals from $\alpha_0$ to $\pi$. 
We used Mathematica for the calculations. We record here that our data for the Chern-Simons invariant of $X_{2n}$ in Table~\ref{table1} and those obtained from  SnapPy match up up to existing decimal points.

\begin{table}
\caption{$\alpha_0$ for $n$ between $1$ and $9$ and for $n$ between $-9$ and $-1$.
\medskip}
 \centering
\vspace{3ex}
\begin{tabular}{ll} 
\noindent\(\begin{array}{|c|c|c|}
\hline
2n & \alpha_0 & \text{\textnormal{cs}} \left(X_{2n}\right)\\
\hline
 2 & 2.5741407781 & 0.155977 \text{} \\
 \hline
 4 & 2.8476422723 & 0. \text{} \\
 \hline
 6 & 2.9424657544 & 0.427829 \text{}  \\
 \hline
 8 & 2.9909391796  & 0.389237 \text{}  \\
 \hline
 10 & 3.0204096324 & 0.365487 \text{}  \\
 \hline
 12 & 3.0402286045 & 0.349444 \text{} \\
 \hline
 14 & 3.0544727854 & 0.337893 \text{}  \\
 \hline
 16 & 3.0652052902 & 0.329184 \text{}   \\
 \hline
 18 & 3.0735826570 & 0.322385 \text{}  \\
 \hline
\end{array}\)
&
\noindent\(\begin{array}{|c|c|c|}
\hline
2n & \alpha_0 & \text{\textnormal{cs}} \left(X_{2n}\right)\\
\hline
 -2 & 2.4071698136 & 0.346796 \text{}  \\
 \hline
 -4 & 2.8082099376 & 0.0217267 \text{}   \\
 \hline
 -6 & 2.9251055596 & 0.100298 \text{}  \\
 \hline
 -8 & 2.9812057191 & 0.141585 \text{}  \\
 \hline
 -10 & 3.0141894961 & 0.166665 \text{}  \\
 \hline
 -12 & 3.0359125478 & 0.183452 \text{}   \\
 \hline
 -14 & 3.0513033433 & 0.195458 \text{}  \\
 \hline
 -16 & 3.0627794492 & 0.204466 \text{}  \\
 \hline
 -18 & 3.0716663560 & 0.211471 \text{}   \\
 \hline
\end{array}\)
\end{tabular}
\label{table1}
\end{table}
%%%%%%%%%%%%%%%%%%%%%%%%%%%%%%%%%%%%%%%%%%%%%%%%%%%%%%%%%%%%%%%%%%
%%%%%%%%%%%%%%%%%%%%%%%%%%%%%%%%%%%%%%%%%%%%%%%%%%%%%%%%%%%%%%%%%%%%%%%%%%

\begin{table} \centering \footnotesize
\caption{Chern-Simons invariant of $X_{2n}$ and of the hyperbolic orbifold, $\text{\textnormal{cs}} \left(X_{2n} (\frac{2 \pi}{k})\right)$ for $n$ between $1$ and $9$ and for $k \left(=\frac{2 \pi}{\alpha}\right)$ between $3$ and $10$, and of its cyclic covering, 
$\text{\textnormal{cs}} \left(M_k (X_{2n})\right)$.
\medskip}
\vspace{3ex}
\begin{tabular}{ll}
\noindent\(\begin{array}{|c|c|c|}
\hline
 k & \text{\textnormal{cs}} \left(X_2(\alpha)\right) & \text{\textnormal{cs}} \left(M_k( X_2)\right) \\
\hline
 3 & 0.0791366 & 0.23741 \\
 4 & 0.105075 & 0.420301 \\
 5 & 0.0215424 & 0.107712 \\
 6 & 0.13151 & 0.789057 \\
 7 & 0.0663635 & 0.464545 \\
 8 & 0.0169609 & 0.135687 \\
 9 & 0.0337443 & 0.303699 \\
 10 & 0.0469426 & 0.469426 \\
\hline
\end{array}\)
&
\noindent\(\begin{array}{|c|c|c|}
\hline
 k &  \text{\textnormal{cs}} \left(X_4(\alpha)\right) & \text{\textnormal{cs}} \left(M_k (X_4)\right)\\
\hline
3 & 0. & 0. \\
 4 & 0. & 0. \\
 5 & 0. & 0. \\
 6 & 0. & 0. \\
 7 & 0. & 0. \\
 8 & 0. & 0. \\
 9 & 0. & 0. \\
 10 & 0. & 0. \\
\hline
\end{array}\)
\end{tabular}

\bigskip

\begin{tabular}{ll}
\noindent\(\begin{array}{|c|c|c|}
\hline
 k &   \text{\textnormal{cs}} \left(X_6(\alpha)\right) & \text{\textnormal{cs}} \left(M_k (X_6)\right) \\
\hline
 3 & 0.125912 & 0.377736 \\
 4 & 0.192764 & 0.771058 \\
 5 & 0.0360431 & 0.180216 \\
 6 & 0.0996796 & 0.598077 \\
 7 & 0.00284328 & 0.0199029 \\
 8 & 0.0554674 & 0.443739 \\
 9 & 0.0409685 & 0.368717 \\
 10 & 0.0294401 & 0.294401 \\
\hline
\end{array}\)
&
\noindent\(\begin{array}{|c|c|c|}
\hline
 k &  \text{\textnormal{cs}} \left(X_{8} (\alpha)\right) & \text{\textnormal{cs}} \left( M_k (X_{8})\right) \\
\hline
 3 & 0.098074 & 0.294222 \\
 4 & 0.157843 & 0.631371 \\
 5 & 0.0993608 & 0.496804 \\
 6 & 0.0622858 & 0.373715 \\
 7 & 0.0365103 & 0.255572 \\
 8 & 0.0174882 & 0.139906 \\
 9 & 0.00284881 & 0.0256393 \\
 10 & 0.091224 & 0.91224 \\
\hline
\end{array}\)
\end{tabular}

\bigskip

\begin{tabular}{ll}
\noindent\(\begin{array}{|c|c|c|}
\hline
 k &  \text{\textnormal{cs}} \left(X_{10} (\alpha)\right) & \text{\textnormal{cs}} \left(M_k (X_{10})\right) \\
\hline
 3 & 0.0781956 & 0.234587 \\
 4 & 0.135356 & 0.541426 \\
 5 & 0.0762905 & 0.381452 \\
 6 & 0.03897 & 0.23382 \\
 7 & 0.0130642 & 0.0914491 \\
 8 & 0.118964 & 0.951709 \\
 9 & 0.0348287 & 0.313458 \\
 10 & 0.067613 & 0.67613 \\
\hline
\end{array}\)
&
\noindent\(\begin{array}{|c|c|c|}
\hline
 k &  \text{\textnormal{cs}} \left(X_{12} (\alpha)\right) & \text{\textnormal{cs}} \left(M_k (X_{12})\right) \\
\hline
 3 & 0.0637865 & 0.191359 \\
 4 & 0.119879 & 0.479518 \\
 5 & 0.0605612 & 0.302806 \\
 6 & 0.0231302 & 0.138781 \\
 7 & 0.068593 & 0.480151 \\
 8 & 0.103028 & 0.82422 \\
 9 & 0.0188687 & 0.169818 \\
 10 & 0.0516364 & 0.516364 \\
\hline
\end{array}\)
\end{tabular}

\bigskip

\begin{tabular}{ll}
\noindent\(\begin{array}{|c|c|c|}
\hline
 k &  \text{\textnormal{cs}} \left(X_{14} (\alpha)\right) & \text{\textnormal{cs}} \left(M_k (X_{14})\right) \\
\hline
 3 & 0.053045 & 0.159135 \\
 4 & 0.108627 & 0.434509 \\
 5 & 0.0491795 & 0.245897 \\
 6 & 0.0116899 & 0.0701393 \\
 7 & 0.0571206 & 0.399844 \\
 8 & 0.0915354 & 0.732283 \\
 9 & 0.00736345 & 0.066271 \\
 10 & 0.0401222 & 0.401222 \\
\hline
\end{array}\)
&
\noindent\(\begin{array}{|c|c|c|}
\hline
 k &  \text{\textnormal{cs}} \left(X_{16} (\alpha)\right) & \text{\textnormal{cs}} \left(M_k (X_{16})\right) \\
\hline
 3 & 0.044788 & 0.134364 \\
 4 & 0.100093 & 0.400374 \\
 5 & 0.0405713 & 0.202857 \\
 6 & 0.00304774 & 0.0182864 \\
 7 & 0.0484584 & 0.339209 \\
 8 & 0.0828601 & 0.662881 \\
 9 & 0.054236 & 0.488124 \\
 10 & 0.0314349 & 0.314349 \\
\hline
\end{array}\)
\end{tabular}

\bigskip

\noindent\(\begin{array}{|c|c|c|}
\hline
 k &  \text{\textnormal{cs}} \left(X_{18} (\alpha)\right) & \text{\textnormal{cs}} \left(M_k (X_{18})\right) \\
\hline
3 & 0.0382642 & 0.114793 \\
 4 & 0.0934054 & 0.373621 \\
 5 & 0.033837 & 0.169185 \\
 6 & 0.162945 & 0.977669 \\
 7 & 0.0416815 & 0.29177 \\
 8 & 0.0760594 & 0.608475 \\
 9 & 0.0474461 & 0.427015 \\
 10 & 0.0246406 & 0.246406 \\
 \hline
\end{array}\)
\label{table2-1}
\end{table}
%%%%%%%%%%%%%%%%%%%%%%%%%%%%%%%%%%%%%%%%%%%%%%%%%%%%%%%%%
\begin{table} \centering \footnotesize
\caption{Chern-Simons invariant of $X_{2n}$ and of the hyperbolic orbifold, $\text{\textnormal{cs}} \left(X_{2n} (\frac{2 \pi}{k})\right)$ for $n$ between $-9$ and $-1$ and for $k \left(=\frac{2 \pi}{\alpha}\right)$ between $3$ and $10$, and of its cyclic covering, $\text{\textnormal{cs}} \left(M_k (X_{2n})\right)$.
\medskip}
\vspace{3ex}
\begin{tabular}{ll}
\noindent\(\begin{array}{|c|c|c|}
\hline
 k & \text{\textnormal{cs}} \left(X_{-2}(\alpha)\right) & \text{\textnormal{cs}} \left(M_k( X_{-2})\right) \\
\hline
 3 & 0.0200144 & 0.0600431 \\
 4 & 0.186811 & 0.747246 \\
 5 & 0.00166667 & 0.00833333 \\
 6 & 0.0504622 & 0.302773 \\
 7 & 0.0163442 & 0.11441 \\
 8 & 0.11699 & 0.935921 \\
 9 & 0.0292902 & 0.263612 \\
 10 & 0.0595432 & 0.595432 \\
\hline
\end{array}\)
&
\noindent\(\begin{array}{|c|c|c|}
\hline
 k &  \text{\textnormal{cs}} \left(X_{-4}(\alpha)\right) & \text{\textnormal{cs}} \left(M_k (X_{-4})\right)\\
\hline
3 & 0.12215 & 0.366451 \\
 4 & 0.0625 & 0.25 \\
 5 & 0.0428241 & 0.21412 \\
 6 & 0.0345888 & 0.207533 \\
 7 & 0.0304384 & 0.213069 \\
 8 & 0.0280513 & 0.22441 \\
 9 & 0.0265452 & 0.238907 \\
 10 & 0.0255297 & 0.255297 \\
\hline
\end{array}\)
\end{tabular}

\bigskip

\begin{tabular}{ll}
\noindent\(\begin{array}{|c|c|c|}
\hline
 k &   \text{\textnormal{cs}} \left(X_{-6}(\alpha)\right) & \text{\textnormal{cs}} \left(M_k (X_{-6})\right) \\
\hline
 3 & 0.0009078 & 0.0027234 \\
 4 & 0.125712 & 0.502846 \\
 5 & 0.0135112 & 0.067556 \\
 6 & 0.108473 & 0.650837 \\
 7 & 0.0344724 & 0.241307 \\
 8 & 0.1044 & 0.835201 \\
 9 & 0.0478866 & 0.430979 \\
 10 & 0.00279032 & 0.0279032 \\
\hline
\end{array}\)
&
\noindent\(\begin{array}{|c|c|c|}
\hline
 k &  \text{\textnormal{cs}} \left(X_{-8} (\alpha)\right) & \text{\textnormal{cs}} \left( M_k (X_{-8})\right) \\
\hline
 3 & 0.0310079 & 0.0930237 \\
 4 & 0.163984 & 0.655935 \\
 5 & 0.0535467 & 0.267733 \\
 6 & 0.149078 & 0.894467 \\
 7 & 0.00389753 & 0.0272827 \\
 8 & 0.0203848 & 0.163079 \\
 9 & 0.0333937 & 0.300543 \\
 10 & 0.0439034 & 0.439034 \\
\hline
\end{array}\)
\end{tabular}

\bigskip

\begin{tabular}{ll}
\noindent\(\begin{array}{|c|c|c|}
\hline
 k &  \text{\textnormal{cs}} \left(X_{-10} (\alpha)\right) & \text{\textnormal{cs}} \left(M_k (X_{-10})\right) \\
\hline
 3 & 0.0524295 & 0.157288 \\
 4 & 0.188321 & 0.753285 \\
 5 & 0.078347 & 0.391735 \\
 6 & 0.0073491 & 0.0440946 \\
 7 & 0.0288918 & 0.202242 \\
 8 & 0.0454073 & 0.363258 \\
 9 & 0.00287658 & 0.0258892 \\
 10 & 0.068952 & 0.68952 \\
\hline
\end{array}\)
&
\noindent\(\begin{array}{|c|c|c|}
\hline
 k &  \text{\textnormal{cs}} \left(X_{-12} (\alpha)\right) & \text{\textnormal{cs}} \left(M_k (X_{-12})\right) \\
\hline
 3 & 0.0678857 & 0.203657 \\
 4 & 0.204874 & 0.819494 \\
 5 & 0.0950581 & 0.47529 \\
 6 & 0.0241032 & 0.144619 \\
 7 & 0.0456616 & 0.319631 \\
 8 & 0.062184 & 0.497472 \\
 9 & 0.0196569 & 0.176912 \\
 10 & 0.0857343 & 0.857343 \\
\hline
\end{array}\)
\end{tabular}

\bigskip

\begin{tabular}{ll}
\noindent\(\begin{array}{|c|c|c|}
\hline
 k &  \text{\textnormal{cs}} \left(X_{-14} (\alpha)\right) & \text{\textnormal{cs}} \left(M_k (X_{-14})\right) \\
\hline
 3 & 0.0793285 & 0.237985 \\
 4 & 0.216795 & 0.867182 \\
 5 & 0.00704512 & 0.0352256 \\
 6 & 0.0361054 & 0.216632 \\
 7 & 0.0576682 & 0.403677 \\
 8 & 0.074192 & 0.593536 \\
 9 & 0.0316653 & 0.284987 \\
 10 & 0.0977426 & 0.977426 \\
\hline
\end{array}\)
&
\noindent\(\begin{array}{|c|c|c|}
\hline
 k &  \text{\textnormal{cs}} \left(X_{-16} (\alpha)\right) & \text{\textnormal{cs}} \left(M_k (X_{-16})\right) \\
\hline
 3 & 0.088066 & 0.264198 \\
 4 & 0.225772 & 0.903087 \\
 5 & 0.0160513 & 0.0802563 \\
 6 & 0.0451171 & 0.270702 \\
 7 & 0.0666807 & 0.466765 \\
 8 & 0.0832028 & 0.665622 \\
 9 & 0.0406762 & 0.366085 \\
 10 & 0.00675282 & 0.0675282 \\
\hline
\end{array}\)
\end{tabular}

\bigskip

\noindent\(\begin{array}{|c|c|c|}
\hline
 k &  \text{\textnormal{cs}} \left(X_{-18} (\alpha)\right) & \text{\textnormal{cs}} \left(M_k (X_{-18})\right) \\
\hline
3 & 0.0949296 & 0.284789 \\
 4 & 0.232766 & 0.931066 \\
 5 & 0.0230607 & 0.115304 \\
 6 & 0.052117 & 0.312702 \\
 7 & 0.00225322 & 0.0157726 \\
 8 & 0.0901914 & 0.721531 \\
 9 & 0.0476676 & 0.429008 \\
 10 & 0.0137386 & 0.137386 \\
 \hline
\end{array}\)
\label{table2-2}
\end{table}
%%%%%%%%%%%%%%%%%%%%%%%%%%%%%%%%%%%%%%%%%%%%%%%%%%%%%%%%%%%%%%%%%%%%%%%%%%

%%%%%%%%%%%%%%%%%%%%%%%%%%%%%%%%%%%%%%%%%%%%%%%%%%%%%%%%%%%%%
%%%%%%%%%%%%%%%%%%%%%%%%%%%%%%%%%%%%%%%%%%%%%%%%%%%%%%%%%%%%
%\clearpage
%\bibliographystyle{plain}
%\bibliography{conereference}
%\printindex
%%%%%%%%%%%%%%%%%%%%%%%%%%%%%%%%%%%%%%%%%%%%%%%%%%%%%%%%%%%%
%%%%%%%%%%%%%%%%%%%%%%%%%%%%%%%%%%%%%%%%%%%%%%%%%%%%%%%
\end{document}